\documentclass[12pt]{amsart}
\usepackage{amssymb}

\begin{document}
\numberwithin{equation}{section}

\def\Label#1{\label{#1}}

\def\1#1{\ov{#1}}
\def\2#1{\widetilde{#1}}
\def\3#1{\mathcal{#1}}
\def\4#1{\widehat{#1}}

\def\s{s}
\def\k{\kappa}
\def\ov{\overline}
\def\span{\text{\rm span}}
\def\tr{\text{\rm tr}}
\def\GL{{\sf GL}}
\def\xo {{x_0}}
\def\Rk{\text{\rm Rk\,}}
\def\sg{\sigma}
\def \emxy{E_{(M,M')}(X,Y)}
\def \semxy{\scrE_{(M,M')}(X,Y)}
\def \jkxy {J^k(X,Y)}
\def \gkxy {G^k(X,Y)}
\def \exy {E(X,Y)}
\def \sexy{\scrE(X,Y)}
\def \hn {holomorphically nondegenerate}
\def\hyp{hypersurface}
\def\prt#1{{\partial \over\partial #1}}
\def\det{{\text{\rm det\,}}}
\def\wob{{w\over B(z)}}
\def\co{\chi_1}
\def\po{p_0}
\def\fb {\bar f}
\def\gb {\bar g}
\def\Fb {\ov F}
\def\Gb {\ov G}
\def\Hb {\ov H}
\def\zb {\bar z}
\def\wb {\bar w}
\def \qb {\bar Q}
\def \t {\tau}
\def\z{\chi}
\def\w{\tau}
\def\Z{\zeta}

\def \T {\theta}
\def \Th {\Theta}
\def \L {\Lambda}
\def\b{\beta}
\def\a{\alpha}
\def\o{\omega}
\def\l{\lambda}
\def \CFT{\text{\rm CFT}}

\def \im{\text{\rm Im }}
\def \re{\text{\rm Re }}
\def \Char{\text{\rm Char }}
\def \supp{\text{\rm supp }}
\def \codim{\text{\rm codim }}
\def \Ht{\text{\rm ht }}
\def \Dt{\text{\rm dt }}
\def \hO{\widehat{\mathcal O}}
\def \cl{\text{\rm cl }}
\def \bR{\mathbb R}
\def \bC{\mathbb C}
\def \bP{\mathbb P}
\def \C{\mathbb C}
\def \bL{\mathbb L}
\def \bZ{\mathbb Z}
\def \bN{\mathbb N}
\def \scrF{\mathcal F}
\def \scrK{\mathcal K}
\def \scrM{\mathcal M}
\def \cR{\mathcal R}
\def \scrJ{\mathcal J}
\def \scrA{\mathcal A}
\def \scrO{\mathcal O}
\def \scrV{\mathcal V}
\def \scrL{\mathcal L}
\def \scrE{\mathcal E}
\def \hol{\text{\rm hol}}
\def \aut{\text{\rm aut}}
\def \Aut{\text{\rm Aut}}
\def \J{\text{\rm Jac}}
\def\jet#1#2{J^{#1}_{#2}}
\def\gp#1{G^{#1}}
\def\gpo{\gp {2k_0}_0}
\def\emmp {\scrF(M,p;M',p')}
\def\rk{\text{\rm rk}}
\def\Orb{\text{\rm Orb\,}}
\def\Exp{\text{\rm Exp\,}}
\def\ess{\text{\rm Ess\,}}
\def\mult{\text{\rm mult\,}}
\def\Jac{\text{\rm Jac\,}}
\def\Span{\text{\rm span\,}}
\def\d{\partial}
\def\D{\3J}
\def\pr{{\rm pr}}
\def\dbl{[\![}
\def\dbr{]\!]}
\def\nl{|\!|}
\def\nr{|\!|}

\def \D{\text{\rm Der}\,}
\def \Rk{\text{\rm Rk}\,}
\def \ima{\text{\rm im}\,}
\def \vfi{\varphi}

\title[CR embeddings of hypersurfaces into hyperquadrics]
{Super-rigidity for CR embeddings of real hypersurfaces into hyperquadrics}
\author[M. S. Baouendi, P. Ebenfelt, X. Huang]
{M. Salah Baouendi, Peter Ebenfelt, Xiaojun Huang}

\address{M. S. Baouendi, P. Ebenfelt: Department of Mathematics, University of
California at San Diego, La Jolla, CA 92093-0112, USA}
\email{sbaouendi@ucsd.edu, pebenfel@math.ucsd.edu}

\address{X. Huang: Department of Mathematics, Rutgers University at New
Brunswick, NJ 08903, USA} \email{huangx@math.rutgers.edu}


\abstract Let $Q^N_l\subset \bC\bP^{N+1}$ denote the standard real,
nondegenerate hyperquadric of signature $l$ and $M\subset \bC^{n+1}$
a real, Levi nondegenerate hypersurface of the same signature $l$.
We shall assume that there is a holomorphic mapping $H_0\colon U\to
\bC\bP^{N_0+1}$, where $U$ is some neighborhood of $M$ in
$\bC^{n+1}$, such that $H_0(M)\subset Q^{N_0}_l$ but
$H(U)\not\subset Q^{N_0}_l$. We show that if $N_0-n<l$ then, for any
$N\geq N_0$, any holomorphic mapping $H\colon U\to \bC\bP^{N+1}$
with $H(M)\subset Q^{N}_l$ and $H(U)\not\subset Q^{N_0}_l$ must be
the standard linear embedding of $Q^{N_0}_l$ into $Q^N_l$ up to
conjugation by automorphisms of $Q^{N_0}_l$ and $Q^N_l$.
\endabstract


\thanks{The authors are supported in part by DMS-0701070,
DMS-0701121, and DMS-0500626.}

\newtheorem{Thm}{Theorem}[section]
\newtheorem{Def}[Thm]{Definition}
\newtheorem{Cor}[Thm]{Corollary}
\newtheorem{Pro}[Thm]{Proposition}
\newtheorem{Lem}[Thm]{Lemma}
\newtheorem{Rem}[Thm]{Remark}
\newtheorem{Ex}[Thm]{Example}
\newtheorem{Con}[Thm]{Conjecture}

\maketitle

\section{Introduction}

It was discovered by Poincar\'e  that
a local non-constant holomorphic mapping sending a piece of the
unit sphere $S$ in $\bC^2$ into itself must in fact be a global
holomorphic automorphism of $\bC\bP^2$ preserving $S$.
Almost fifty years later, Alexander \cite{A} completed the Poincar\'e's program along these lines in the equal-dimensional case,
by showing that a continuous non-constant CR map from an open piece of the unit sphere $S$ in $\bC^n$ into  $S$ for any $n\ge 2$
is also a automorphism of $\bC\bP^n$ preserving $S$.

Webster \cite{We79} first obtained a similar rigidity result  for holomorphic mappings (or sufficient smooth CR mappings) sending a
piece of the unit sphere $S^{n}$ in $\bC^{n+1}$ into the unit sphere
$S^{N}$ in a different complex space $\bC^{N+1}$ with $N=n+1\ge 3$. 
Cima-Suffridge in \cite{CS83}  conjectured that
the just mentioned Poincar\'e-Alexander-Webster rigidity property holds for any $C^2$-smooth non-constant CR map,
provided that the codimension $N-n<n$.  This was verified by Faran in \cite{Faran86} when the map is real analytic.
 Forstneric's reflection principle in \cite{Fo86} shows that it  holds when the map is $C^{N-n+1}$-smooth.
 In \cite{Hu1}, this  super rigidity was finally established  for any non-constant $C^2$-smooth CR map.
 It is not clear if one can go below $C^2$-smoothness
to obtain the same result in \cite{Hu1}. However, the  development of inner function theory
demonstrates that, in the sharp contrary to the equi-dimensional case, the  theorem in \cite{Hu1}
does not hold for  general continuous CR mappings.
The bound $N<2n$ is optimal as can be seen by examples such as the so-called Whitney map
(see e.g.\ Example 1.1 in \cite{EHZ1}); the reader is also referred
to \cite{Fa82} and \cite{HJ} \cite{Ha05} and \cite{HJX} for a classification of all rational maps in the case
$n=1,N=2$ or in case  $N\le  3n-4$.

 The situation is quite different in the case of maps between nondegenerate {\it pseudoconcave}
 hyperquadrics.
An immediate benefit from the Lewy extension theorem in this
consideration is that  one needs only to deal with holomorphic maps
instead of more subtle CR maps. More recently, It was  shown in
 \cite{BH} that for such hyperquadrics there is no
restriction on the codimension $N-n$ for the analogous rigidity
phenomenon to hold. In the present paper, we study a more general
situation where the source manifold is not necessarily a
hyperquadric. We consider holomorphic mappings sending a given
Levi-nondegenerate pseudoconcave hypersurface $M$ in $\bC^{n+1}$ into a
nondegenerate hyperquadric of the same signature in $\bC\bP^{N+1}$
and show that if $M$ is sufficiently close to a hyperquadric in a
certain sense,  then any two such mappings differ only by an
automorphism of the hyperquadric (see Theorem \ref{t:mainhyper} for
the precise formulation). Previous results along these lines in the strictly pseudoconvex case include
\cite{We79}, \cite{EHZ1}, and in the general case \cite{EHZ2}. The proof of our main result relies  on the early work in the study  of Pseudo-Hermitian geometry (see \cite {We78} \cite {We79} \cite  {Le88} and the references therein) and, in particular, the more  recent
derivations in \cite{EHZ1} and \cite{EHZ2}.

Let $M\subset \bC^{n+1}$ be a smooth hypersurface and $p\in M$.
Assume that $M$ is Levi nondegenerate at $p$ and
$\mathcal L\colon \bC^n\times \bC^n\to \bC$ a representative of the Levi form of $M$ at $p$.
If we let $e_-$ and $e_+$ be the number of negative and positive eigenvalues of $\mathcal L$,
respectively, then $l(M,p):=\min(e_-,e_+)\leq n/2$ is independent of the choice of
representative $\mathcal L$ of the Levi form. We shall refer to $l(M,p)$ as the signature of
$M$ at $p$.  If $M$ is connected and Levi nondegenerate at every point, then $l:=l(M,p)$ is
constant and we shall say that $M$ has signature $l$.

We let $Q^N_l\subset\bC\bP ^{N+1}$ denote the standard hyperquadric
of signature $0\leq l\leq N/2$ given in homogeneous coordinates
$[z_0\colon z_1\colon\ldots\colon z_{N+1}]$ by
\begin{equation}
\Label{hyperquadric}
-\sum_{j=0}^l|z_j|^2+\sum_{k=l+1}^{N+1}|z_k|^2=0.
\end{equation}
(Thus, the superscript in $Q^N_l$ represents the CR dimension and
the subscript represents  the signature.) We observe that $Q^N_l$ is
a connected hypersurface of CR dimension $N$, which is Levi
nondegenerate at every point. Its signature is $l$. We denote by
$\Aut(Q^N_l)$ the subgroup of biholomorphic mappings of
$\bC\bP^{N+1}$ preserving $Q^N_l$. It is well known \cite{CM} that
$\Aut(Q^N_l)$ can be identified with the group of invertible
$(N+2)\times(N+2)$ matrices that preserve the quadratic form on the
left hand side of \eqref{hyperquadric} (up to sign if $l=N/2$). We
also note that if $2l \le N_0<N$, then the standard linear embedding
$L\colon \bC\bP^{N_0+1}\to \bC\bP^{N+1}$, given by
\begin{equation}
\Label{linear}
L([z_0\colon\ldots\colon z_{N_0+1}]):=[z_0\colon \ldots\colon z_{N_0+1}\colon 0\colon\ldots\colon 0],
\end{equation}
satisfies $L(Q^{N_0}_l)\subset Q^{N}_l$.

To formulate our main result, we shall need one more definition. If $M\subset \bC^{n+1}$ is a real hypersurface, then we shall say that $M$ is {\it locally  biholomorphically equivalent}  to the hyperquadric $Q^{n}_{l}$ at $p\in M$ if there are $p'\in Q^n_l$, open neighborhoods $U\subset \bC^{n+1}$ and $V\subset \bC\bP^{n+1}$ of $p$ and $p'$, respectively, and a biholomorphism $H\colon U \to V$ such that $H(M\cap U)= Q^n_{l}\cap V$ and $H(p)=p'$.
Our main result is the following.

\begin{Thm}\Label{t:mainhyper}
Let $M\subset \bC^{n+1}$ be a connected real-analytic
Levi-nondegenerate hypersurface of signature $l\leq n/2$. Moreover, if $l=n/2$, then  assume that $M$ is not locally biholomorphically equivalent to the hyperquadric $Q^{n}_{n/2}$ at any point of $M$. Suppose
that there is an open connected neighborhood $U$ of $M$ in
$\bC^{n+1}$ and a holomorphic mapping $f_0\colon U\to \bC\bP^{N_0+1}$
with $f_0(M)\subset Q^{N_0}_l$ such that $f_0(U)\not\subset
Q^{N_0}_l$. If $f\colon U\to \bC\bP^{N+1}$ is a holomorphic mapping
with $f(M)\subset Q^{N}_l$, $f(U)\not\subset Q^{N}_l$,  and
$N_0-n<l$, then there is $T\in
\Aut(Q^{N}_l)$ such that $f:=T\circ L\circ f_0$, where $L$
denotes the standard linear embedding given by \eqref{linear}.
\end{Thm}

The conclusion of Theorem \ref{t:mainhyper} with
the additional assumption that $M$ is the hyperquadric $Q^n_l$ (and
$N_0=n$, $f_0(z)\equiv z$) is contained in Theorem 1.6 (i) of \cite{BH}. If
the condition $N_0-n<l$ is replaced by $N_0+N<3n$, then the
conclusion of Theorem \ref{t:mainhyper} follows from the work
\cite{EHZ1} (in the strictly pseudoconvex case $l=0$) and
\cite{EHZ2} (in the general case). We conclude the introduction with a number of remarks.

\begin{Rem} {\rm We point out that if $M\subset\bC^{n+1}$ is a merely smooth ($C^\infty$)
connected Levi-nondegenerate hypersurface of signature $l>0$ and
$F\colon M\to Q^N_l\subset\bC\bP^{N+1}$ a smooth CR mapping, then
$F$ is the restriction to $M$ of a holomorphic mapping $f\colon U\to
\bC\bP^{N+1}$, where $U$ is an open neighborhood of $M$ in
$\bC^{n+1}$. Indeed, this follows from a classical result of Lewy
\cite{Lewy} (see also  Theorem 2.6.13 in \cite{Ho}), since the Levi
form of $M$ has eigenvalues of both signs at every point. If, in
addition, $f(U)$ is not contained in $Q^N_l$, then $M$ is
real-analytic. To see this, let $p_0$ be a point on $M$ and $\rho=0$
a real-analytic defining equation for $Q^N_l$ (in some local chart)
near $f(p_0)$. It follows that $M$ is contained, near $p_0$, in the
real-analytic variety $V$ defined by $\rho\circ f=0$. Since
$f(U)\not\subset Q^N_l$, it follows that $\rho\circ f\not\equiv 0$
and hence $V$ is non-trivial. The real-analyticy of $M$ now follows
from a theorem of Malgrange \cite{Mal}. Hence, the conditions in
Theorem \ref{t:mainhyper} that $M$ is real-analytic and $f_0,f$ are
holomorphic can be weakened to $M$ being smooth and $f_0,f$ being CR
with the appropriate conditions on their holomorphic extensions. }
\end{Rem}

\begin{Rem}\Label{lbeq}
{\rm We also remark that if $M\subset \bC^{n+1}$ is a connected
real-analytic Levi-nondegenerate hypersurface of signature $l$ and $M$ is locally
biholomorphically equivalent to the hyperquadric $Q^n_l$ at some point $p\in M$, then $M$ is
locally biholomorphically equivalent to $Q^n_l$ at every point in $M$. Indeed, this follows from
the fact that $M$ is locally biholomorphically equivalent to $Q^n_l$ at $p$ if and only if the
CR curvature of $M$ (see below) vanishes identically in an open neighborhood of $p$ in $M$. The
conclusion above now follows from the real-analyticy of the CR curvature of $M$ and the
connectedness of $M$. Hence, the additional assumption in Theorem \ref{t:mainhyper}
when $l=n/2$ that $M$ is not locally biholomorphically equivalent to  $Q^n_{n/2}$ at any point
of $M$ can be replaced by the seemingly weaker condition that $M$ is not locally
biholomorphically equivalent to  $Q^n_{n/2}$ at one point in $M$.
}
\end{Rem}

\begin{Rem} {\rm If $M$ is locally biholomorphically equivalent to  $Q^n_{n/2}$ at some point
$p\in M$ (and hence at every point of $M$ by Remark \ref{lbeq}), then the conclusion of Theorem
\ref{t:mainhyper} does not hold in general. However, the situation can be reduced to one
considered in \cite{BH} as follows. Under the assumption above, we may take $N_0=n$ in the
statement of Theorem  \ref{t:mainhyper} and, by shrinking $U$ if necessary, we may assume that
$f_0\colon U\to \bC\bP^{n+1}$ is a biholomorphism (onto its image) sending $M$ into $Q^{n}_{n/2}$. Let $f$ be as in the statement of Theorem \ref{t:mainhyper}. By applying Theorem 1.6 in \cite{BH} to the mapping $f\circ f_0^{-1}$, we conclude that $f=T\circ L \circ T_0\circ f_0$, where $T$ and $L$ are as in Theorem \ref{t:mainhyper} and $T_0$ is either the identity in $\bC\bP^{n+1}$ or the flip
\begin{equation}\Label{flip}
[z_0\colon z_1\colon\ldots\colon z_n\colon z_{n+1}]\mapsto
[z_{n+1}\colon z_{n}\colon\ldots z_{1}\colon z_{0}].
\end{equation}
We note that it is not always possible to take $T_0$ to be the identity in this situation.
}
\end{Rem}

\begin{Rem} {\rm If there is an open connected neighborhood $U$ of $M$ in $\bC^{n+1}$ and a
holomorphic mapping $f_0\colon U\to \bC\bP^{N_0+1}$ with
$f_0(M)\subset Q^{N_0}_l$ such that $f_0(U)\not\subset Q^{N_0}_l$,
then necessarily $N_0\geq n$. Indeed, if $N_0<n$, then the rank of
$f_0$ would be $\leq n$ at every point of $M$. Theorem 5.1 in
\cite{BERtrans} would then imply that $f_0(U)\subset Q^{N_0}_l$
contradicting the hypothesis above. }
\end{Rem}

\section{Two basic lemmas}

In this section, we shall formulate two lemmas that are key
ingredients in the proof of Theorem \ref{t:mainhyper}. The first
lemma was proved by  \cite{Hu1} and \cite{EHZ2} ([Lemma 3.2, \cite{Hu1}]). For the reader's convenience, we reproduce its
statement here.

\begin{Lem}\Label{l:main0}
 Let $k,l,n$ be nonnegative integers such $1\leq k<n$. Assume that
$g_1,\ldots, g_k$, $f_1,\ldots, f_k$ are germs at $0\in \bC^n$ of
holomorphic functions such that
\begin{equation}\Label{e:basiceq}
\sum_{i=1}^kg_i(z)\overline{f_j(z)}=
A(z,\bar z)\bigg( -\sum_{i=1}^l|z_i|^2+\sum_{j=l+1}^n|z_j|^2\bigg),
\end{equation}
where $A(z,\zeta)$ is a germ at $0\in \bC^n\times\bC^n$ of a
holomorphic function. Then $A(z,\bar z)\equiv 0$.
\end{Lem}

In \cite{Hu1}, Lemma \ref{l:main0} is stated only for $l=0$, but the
proof for $l>0$ is identical (see Lemma 2.1 in \cite{EHZ2}).
 Lemma \ref{l:main0} was also a crucial
tool in the papers \cite{Hu1}, \cite{EHZ1}, \cite{EHZ2}. The second lemma that we shall need is the following.

\begin{Lem}\Label{l:main}
Let $k,l,n$ be nonnegative integers such that
$k<l\leq n/2$. Assume that $g_1,\ldots, g_k,f_1,\ldots, f_m$ are
germs at $0\in \bC^n$ of holomorphic functions such that
\begin{equation}\Label{e:basiceq2}
-\sum_{i=1}^k|g_i(z)|^2+\sum_{j=1}^m|f_j(z)|^2=
A(z,\bar z)\bigg( -\sum_{i=1}^l|z_i|^2+\sum_{j=l+1}^n|z_j|^2\bigg),
\end{equation}
where $A(z,\zeta)$ is a germ at $0\in \bC^n\times\bC^n$ of a
holomorphic function. Then $A(z,\bar z)\equiv 0$.
\end{Lem}

The proof of Lemma \ref{l:main} can be found in Lemma 4.1 of
\cite{BH} (with $\ell'=\ell$ and after a direct application of Lemma
2.1 of \cite{BH}). The lemma also follows in a straightforward way
from Theorem 5.7 in the subsequent work \cite{BERtrans}.

\section{Preliminaries}\Label{s:prelim}

We shall use the set-up and notation of \cite{EHZ1}. The reader is
referred to that paper for the terminology used below and a brief
introduction to the pseudohermitian geometry   and
the CR pseudoconformal geometry. (The reader is of course
also referred to the original papers by Chern and Moser \cite{CM},
Webster \cite{We78}, and Tanaka \cite{Ta75}.) Although the main
focus of \cite{EHZ1} is on strictly pseudoconvex hypersurfaces, many
of the results obtained in that paper work equally well for Levi-nondegenerate hypersurfaces and we shall use those results in this
paper. Thus, let $M$ be a Levi-nondegenerate CR-manifold of
dimension $2n+1$, with rank $n$ CR bundle $\mathcal V$, and
signature $l\leq n/2$. Near a distinguished point $p_0\in M$, we let
$\theta$ be a contact form and $T$ its characteristic (or Reeb)
vector field, i.e.\ the unique real vector field that satisfies
$$
T\lrcorner d\theta=0,\quad \big<\theta,T\big >=1.
$$
We complete $\theta$ to an admissible coframe
$(\theta,\theta^1,\ldots,\theta^n)$ for the bundle $T'M$ of
$(1,0)$-cotangent vectors (i.e. the cotangent vectors that annihilate $\mathcal V$). Recall that the coframe is called
admissible if $\big<\theta^\alpha,T\big>=0$, for $\alpha=1,\ldots,
n$. We choose a frame $L_1,\ldots L_n$ for the bundle $\bar
{\mathcal V}$, or, as we shall also refer to it, the bundle of
$(1,0)$-tangent vectors $T^{1,0}M$. The frame for $T^{1,0}M$ will be chosen such
that $(T,L_1,\ldots, L_n, L_{\bar 1},\ldots L_{\bar n})$ is a frame
for $\bC TM$, near $p_0$, which is dual to the coframe
$(\theta,\theta^1,\ldots, \theta^n,\theta^{\bar
1},\ldots,\theta^{\bar n})$. Here and in what follows,
$L_{\bar\alpha}=\overline{L_\alpha}$, $\theta^{\bar\alpha}=\overline{\theta^\alpha}$, etc. We shall denote the
matrix representing the Levi form (relative to the frame $L_1,\ldots
L_n$) by $(g_{\alpha\bar\beta})$, where $\alpha,\beta=1,\ldots, n$.
We may assume that  $g_{\alpha\bar\beta}$ is constant, in fact
that it is diagonal with diagonal elements $-1,\ldots, -1$ ($l$
times) and $1,\ldots, 1$ ($n-l$ times), although this fact will not be explicitly used most of the time. We denote by $\nabla$ the
Webster--Tanaka pseudohermitian connection on $\bar {\mathcal V}$,
which is expressed relative to the chosen frame and coframe
 by
\begin{equation}\Label{eq-con}
\nabla L_\a:={\o_\a}^\b\otimes L_\b,
\end{equation}
where the $1$-forms ${\o_\a}^\b$ on $M$ are uniquely determined by
the conditions
\begin{equation}\Label{eq-consymmetry}
\begin{aligned}
d\theta^\b&=\theta^\a\wedge {\o_\a}^\b \mod \theta\wedge\theta^{\bar\a},\\
dg_{\a\bar\b}&=\o_{\a\bar\b}+ \o_{\bar\b\a}.
\end{aligned}
\end{equation}
Here and for the remainder of this paper, we use the summation
convention that an index that appears both as a subscript and
superscript is summed over. We also use the Levi form to raise and
lower indices in the usual way. The first condition in
\eqref{eq-consymmetry} can be rewritten as
\begin{equation}\Label{eq-conmatrix}
d\theta^\b=\theta^\a\wedge {\o_\a}^\b + \theta\wedge\tau^\b, \quad
\tau^\b = A^\b{}_{\bar\nu}\theta^{\bar\nu}, \quad A^{\a\b} =
A^{\b\a}
\end{equation}
for a suitable uniquely determined torsion matrix $(A^\b{}_{\bar\a})$,
where the last symmetry relation holds automatically (see
\cite{We78}). For future reference, we record here also the fact
that the coframe $(\theta,\theta^1,\ldots,\theta^n)$ is admissible
if and only if
\begin{equation}\Label{Levi}
d\theta=ig_{\alpha\bar\beta}\theta^\alpha\wedge\theta^{\bar\beta}.
\end{equation}

Now, let $\hat M$ be a Levi-nondegenerate CR-manifold of dimension
$2\hat n+1$, with rank $\hat n$ CR bundle $\hat{\mathcal V}$
($=\overline{T^{1,0}\hat M}$), and signature $\hat l \leq \hat n/2$.
Let $f\colon M\to \hat M$ be a smooth CR mapping. Our arguments in the sequel will be of a local nature and we shall restrict our attention to a small open neighborhood  of $p_0$ (that we still shall refer to as $M$). We shall use a $\hat{}$ to denote
objects associated to $\hat M$. Capital Latin indices $A, B,$ etc,
will run over the set $\{1,2,\ldots,{\hat n}\}$ whereas Greek
indices $\a,\b$, etc, will run over $\{1,2,\ldots, n\}$ as above.
Moreover, we shall let small Latin indices $a,b,$ etc, run over the
complementary set $\{n+1,n+2,\ldots, {\hat n}\}$.
Recall that $f\colon M\to \hat M$ is a CR mapping if
\begin{equation}\Label{CR}
f^*(\hat\theta)=a\theta,\quad f^*(\hat\theta^A)=E^A{}_\alpha\theta^\alpha+E^A\theta,
\end{equation}
where $a$ is a real-valued function and $E^A{}_\alpha$, $E^A$ are complex-valued functions defined near $p_0$. We shall assume that $f$ is {\it CR
transversal} to $\hat M$ at $p_0\in M$, which in the present context
can be expressed by saying that $a(p_0)\neq 0$, where $a$ is the function in \eqref{CR}. Without loss of generality, we may assume that $a\equiv 1$ (i.e.\ we take $\theta=f^*(\hat\theta)$ in our admissible coframe $(\theta,\theta^\alpha)$). We note that the CR transversality of $f$ implies that $n\leq \hat n$. Indeed, it follows easily from \eqref{Levi} and \eqref{CR} that
\begin{equation}\Label{Levi-id}
g_{\alpha\bar\beta}=\hat g_{A\bar B} E^A_\alpha E^{\bar B}_{\beta}.
\end{equation}
Since the rank of the matrices $(g_{\alpha\bar\beta})$ and $(\hat
g_{A\bar B})$ are $n$ and $\hat n$, respectively, we conclude that
$n\leq \hat n$ and the rank of the matrix $(E^A{}_{\alpha})$ is $n$.
Hence, if $f$ is CR tranversal to $\hat M$, it also follows that $f$
is an embedding, locally near $p_0$.  We may assume, without loss of
generality (by renumbering the $\hat\theta^A$ if necessary), that
the admissible coframe $(\hat\theta,\hat\theta^A)$ on $\hat M$ is
such that the pullback
$(\theta,\theta^\alpha):=(f^*(\hat\theta),f^*(\hat\theta^\a))$ is a
coframe for $M$. Assume that
$(\theta,\theta^\alpha)$ defined in this way is also admissible.
Hence, we shall drop the $\hat{}$ over the frames and coframes if
there is no ambiguity. It will be clear from the context if a form
is pulled back to $M$ or not. Under the assumptions above, we shall
identify $M$ with the submanifold $f(M)$ of $\hat M$ and write
$M\subset\hat M$. Then $T^{1,0}M$ becomes a rank $n$ subbundle of
$T^{1,0}\hat M$ along $M$. It follows that the (real) codimension of
$M$ in $\hat M$ is $2({\hat n}-n)$ and that there is a rank $({\hat
n}-n)$ subbundle $N'M$ of $T'\hat M$ along $M$ consisting of
$1$-forms on $\hat M$ whose pullbacks to $M$ (under $f$) vanish. We
shall call $N'M$ the {\em holomorphic conormal bundle of $M$} in
$\hat M$. We shall say that the pseudohermitian structure $(\hat
M,\hat\theta)$ (or simply $\hat\theta$) is {\em admissible for the
pair $(M,\hat M)$} if the characteristic vector field $\hat T$ of
$\hat \theta$ is tangent to $M$ (and hence its restriction to $M$
coincides with the characteristic vector field $T$ of $\theta$). If
the admissible coframe $(\hat\theta,\hat\theta^A)$ on $\hat M$ is
such that $(\theta,\theta^\alpha)$, with $\theta:=f^*(\hat\theta)$,
$\theta^\alpha:=f^*(\hat\theta^\alpha)$, is an admissible coframe on
$M$ and $f^*(\hat\theta^a)=0$, then $(\hat M,\hat\theta)$ is
admissible for the pair $(M,\hat M)$.

It is easily seen that not all contact forms $\hat\theta$ are
admissible for $(M,\hat M)$. However, Lemma 4.1 in \cite{EHZ1}
(which, though stated only for strictly pseudoconvex CR-manifolds,
holds also for general Levi nondegenerate CR-manifolds) asserts that
any contact form $\theta$ on $M$ can be extended to a contact form
$\hat \theta$ in a neighborhood of $M$ in $\hat M$ such that $\hat
\theta$ is admissible for $(M,\hat M)$. Let us fix a contact form
$\theta$ on $M$, extend it to an admissible contact form
$\hat\theta$ for the pair $(M,\hat M)$. We denote by $\hat T$ the
characteristic vector field of $\hat \theta$ and by $T$ its
restriction to $M$. Recall that $T^{1,0}M$ is a rank $n$ subbundle
of the rank $\hat n$ bundle $T^{1,0}\hat M$. The Levi form of $M$ at
a point $p\in M\subset \hat M$ is given, under these
identifications, by the restriction of the Levi form of $\hat M$ to
the subspace $T^{1,0}_pM\subset T^{1,0}_p\hat M$ (and, hence, in
particular, $\hat l\geq l$). If we let $(L_\alpha)$ be a frame for
$T^{1,0}M$ such that the Levi form $g_{\alpha\bar\beta}$ of $M$ is
constant and diagonal with $-1,\ldots,-1$ ($l$ times) and $1,\ldots,
1$ ($n-l$ times) on the diagonal, then we may complete $(L_\alpha)$
to a frame $(\hat L_A)=(L_\alpha, \hat L_a)$ for $T^{1,0}\hat M$
along $M$ such that the Levi form $\hat g_{A\bar B}$ of $\hat M$
along $M$ is constant and diagonal with diagonal elements
$-1,\ldots,-1$ ($l$ times), $1,\ldots, 1$ ($n-l$ times), $-1,\ldots,
-1$ ($\hat l-l$ times) and $1,\ldots, 1$ ($\hat n-n-\hat l+l$
times). Finally, we extend the $\hat L_A$ to a neighborhood of $M$
such that the Levi form of $\hat M$ stays constant. If we now let
$(\hat\theta,\hat\theta^A,\hat\theta^{\bar A})$ be the dual coframe
of $(\hat T, \hat L_A,\hat L_{\bar A})$, then clearly the coframe
$(\hat\theta,\hat\theta^A)$ for $T'\hat M$ is admissible, its
pullback to $M$ equals $(\theta,\theta^\alpha,0)$ and
$(\theta,\theta^\alpha)$ is an admissible coframe for $T'M$.
 In
other words, we have obtained the following result, in whose
formulation we have taken a little more care to distinguish between
$M$ and its image $f(M)$ in $\hat M$. A similar result was obtained
in \cite{EHZ1} (Corollary 4.2) for strictly pseudoconvex
hypersurfaces.

\begin{Pro}\Label{thm-admada}
Let $M$ and $\hat M$ be Levi-nondegenererate CR-manifolds of
dimensions $2n+1$ and $2{\hat n}+1$, and signatures $l\leq n/2$ and
$\hat l\leq \hat n/2$, respectively. Let $f\colon M\to \hat M$ be a
CR mapping that is CR transversal to $\hat M$ along $M$. If
$(\theta,\theta^\a)$ is any admissible coframe on $M$, then in a
neighborhood of any point $\hat p\in f(M)$ in $\hat M$ there exists
an admissible coframe $(\hat \theta,\hat \theta^A)$ on $\hat M$ with
$f^*(\hat\theta,\hat \theta^\a,\hat\theta^a)=(\theta,\theta^\a,0)$.
In particular, $\hat\theta$ is admissible for the pair $(f(M),\hat
M)$, i.e.\ the characteristic vector field $\hat T$ is tangent to
$f(M)$. If the Levi form of $M$ with respect to $(\theta,\theta^\a)$
is constant and diagonal with $-1,\ldots,-1$ ($l$ times) and
$1,\ldots, 1$ ($n-l$ times) on the diagonal, then $(\hat \theta,\hat
\theta^A)$ can be chosen such that the Levi form of $\hat M$
relative to this coframe is constant and diagonal with diagonal elements
$-1,\ldots,-1$ ($l$ times), $1,\ldots, 1$ ($n-l$ times), $-1,\ldots,
-1$ ($\hat l-l$ times) and $1,\ldots, 1$ ($\hat n-n-\hat l+l$ times).
 With this additional property, the coframe $(\hat \theta,\hat
\theta^A)$ is uniquely determined along $M$ up to unitary
transformations in $U(n,l)\times U(\hat n-n, \hat l-l)$.
\end{Pro}

Let us fix an admissible coframe $(\theta,\theta^\a)$ on $M$ and let
$(\hat \theta,\hat \theta^A)$ be an admissible coframe on $\hat M$
near a point $\hat p\in f(M)$. We shall say that $(\hat \theta,\hat \theta^A)$ is {\em
adapted} to $(\theta,\theta^\a)$ on $M$ (or simply to $M$ if the
coframe on $M$ is understood) if it satisfies the conclusion of
Proposition~{\rm\ref{thm-admada}} with the requirement there for the
Levi form. For convenience of notation though, we continue to denote
the Levi forms by $g_{\alpha\bar\beta}$ and $\hat g_{A\bar B}$.

For ease of notation, we shall write $(\theta,\theta^A)$ for the coframe $(\hat \theta,\hat \theta^A)$.
The fact that $(\theta,\theta^A)$  is
adapted to $M$ implies, in view of (\ref{eq-conmatrix}), that if the
pseudohermitian connection matrix of $(\hat M,\hat{\theta})$ is
$\hat\omega_B{}^A$, then that of $(M,\theta)$ is (the pullback of)
$\hat\omega_\b{}^\a$. Similarly, the pulled back torsion
$\hat\tau^\a$ is $\tau^\a$. Hence omitting a $\hat{}$ over these
pullbacks will not cause any ambiguity and we shall do it in the
sequel.
By our normalization of the Levi form, the second equation in
(\ref{eq-consymmetry}) reduces to
\begin{equation}\Label{eq-consym2}
 \omega_{B\bar A}+ \omega_{\bar AB}=0,
\end{equation}
where as before $\o_{\bar AB}=\1{ \o_{A\bar B}}$.

The matrix of $1$-forms $({\omega_\a}^b)$ pulled back to $M$ defines
the {\it second fundamental form} of $M$ (or more precisely of the
embedding $f$). Since $\theta^b$ is 0 on $M$, we deduce by using
equation (\ref{eq-conmatrix}) that, on $M$,
\begin{equation}
{{\omega_\a}^b}\wedge \theta^\a + \tau^b\wedge\theta=0,
\end{equation} which
implies that
\begin{equation}\Label{eq-secform}
{\omega_\a}^b={{\omega_\a}^b}_\b\,\theta^\b, \quad
\omega_\a{}^b{}_\b=\omega_\b{}^b{}_\a, \quad \tau^b=0.
\end{equation}
As in \cite{EHZ1}, we identify the CR-normal space $T^{1,0}_p\hat
M/T^{1,0}_pM$ with $\bC^{\hat n-n}$ by letting the equivalence
classes of the $L_a$ form a basis in the former space. We consider
the components of the second fundamental form
$({{\omega_\alpha}^a}_\beta)_{a=n+1,\ldots,\hat
n}={{\omega_\alpha}^a}_\beta L_a$, for $\alpha,\beta=1,\ldots,n$, as
vectors in the CR-normal space $\cong\bC^{\hat n-n}$. We also view
the second fundamental form ${{\omega_\alpha}^a}_\beta$ as a section
over $M$ of the vector bundle of $\C$-bilinear maps
$$T^{1,0}_pM\times T^{1,0}_pM\to T^{1,0}_p\hat M/T^{1,0}_pM, \quad p\in M.$$
For sections of this bundle we have the covariant differential
induced by the pseudohermitian connections $\nabla$ and $\hat\nabla$
on $M$ and $\hat M$ respectively:
\begin{equation}\Label{o-der}
\nabla\omega_{\a}{}^a{}_\b =d\omega_\a{}^a{}_\b
-\omega_\mu{}^a{}_\b\,\o_\a{}^\mu +\omega_\a{}^b{}_\b\,\o_b{}^a-
\omega_\a{}^a{}_\mu\,\o_\b{}^\mu.
\end{equation}
We use e.g.\ $\o_\a{}^a{}_{\b;\gamma}$ to denote its component in
the direction $\theta^\gamma$. Higher order covariant derivatives
$\o_\a{}^a{}_{\b;\gamma_1,\ldots,\gamma_l}$ are defined inductively
in a similar way:
\begin{equation}\Label{eq-cov1}
\nabla\omega_{\gamma_1}{}^a{}_{\gamma_2;\gamma_3\ldots\gamma_j}=
d\omega_{\gamma_1}{}^a{}_{\gamma_2;\gamma_3\ldots\gamma_j}+
\omega_{\gamma_1}{}^b{}_{\gamma_2;\gamma_3\ldots\gamma_j} \,
\o_b{}^a -\sum_{l=1}^j
\omega_{\gamma_1}{}^a{}_{\gamma_2;\gamma_3\ldots\gamma_{l-1}\,\mu\,
\gamma_{l+1}\ldots\gamma_j} \,\o_{\gamma_l}{}^\mu.
\end{equation}
 As above, we also consider the covariant derivatives as vectors in
$\bC^{\hat n-n}\cong T^{1,0}_p\hat M/T^{1,0}_pM$ via the identification
$$
(\omega_{\gamma_1}{}^a{}_{\gamma_2;\gamma_3\ldots\gamma_j})_{a=n+1,\ldots,\hat
n}=\omega_{\gamma_1}{}^a{}_{\gamma_2;\gamma_3\ldots\gamma_j}L_a.
$$
We define an increasing sequence of vector spaces
$$E_2(p)\subset\ldots \subset E_k(p)\subset\ldots\subset \bC^{\hat n-n}\cong
T^{1,0}_p\hat M/T^{1,0}_pM$$ by letting $E_k(p)$ be the span of the
vectors
$$(\omega_{\gamma_1}{}^a{}_{\gamma_2;\gamma_3\ldots\gamma_j})_{a=n+1,\ldots,\hat
n},\quad \forall\, 2\leq j\leq k,\ \gamma_i\in\{1,\ldots,n\},
$$
evaluated at $p\in M$. We shall say that the mapping $f\colon M\to
\hat M$ is {\it constantly $(k,s)$-degenerate} at $p$ (following Lamel
\cite{Lamel}, see \cite{EHZ1}) if the vector space $E_k(q)$ has
constant dimension $\hat n-n-s$ for $q$ in an open neighborhood of $p$,
$E_{k+1}(q)= E_k(q)$, and $k$ is the smallest integer with
this property.

\section{The second fundamental form, covariant derivatives, and the Gauss equation}\Label{s:SFF}

For the proof of our main results, we need to recall some further
results and terminology from \cite{EHZ1}. We keep the notation from
the previous section. A tensor
$T_{\a_1\ldots\a_r\bar\b_1\ldots\bar\b_s}{}^{a_1\ldots a_t\bar
b_1\ldots\bar b_q}$, with $r,s\geq1$, is called  {\em conformally flat} if it is a
linear combination of $g_{\a_i\bar\b_j}$ for $i=1,\ldots,r$,
$j=1,\ldots,s$, i.e.
\begin{equation}
T_{\a_1\ldots\a_r\bar\b_1\ldots\bar\b_s}{}^{a_1\ldots a_t\bar
b_1\ldots\bar b_q}=\sum_{i=1}^r\sum_{j=1}^s g_{\alpha_i\bar\beta_j}
(T_{ij})_{\a_1\ldots\widehat{\a_i}\ldots
\a_r\bar\b_1\ldots\widehat{\bar\b_j}\ldots\ldots\bar\b_s}{}^{a_1\ldots
a_t\bar b_1\ldots\bar b_q},
\end{equation}
where e.g.\ $\widehat{\a}$ means omission of that factor. (A similar definition can be made for tensors with different
orderings of indices.) The following observation gives a motivation for this definition. Let $T_{\a_1\ldots\a_r\bar\b_1\ldots\bar\b_s}{}^{a_1\ldots a_t\bar b_1\ldots\bar b_q}$ be a tensor, symmetric in $\alpha_1,\ldots,\alpha_r$ as well as in $\beta_1,\ldots,\beta_s$, and form the homogeneous vector-valued polynomial of type $(r,s)$ whose components are given by $$T^{a_1\ldots a_t\bar b_1\ldots\bar b_q}(\zeta,\bar \zeta):=
T_{\a_1\ldots\a_r\bar\b_1\ldots\bar\b_s}{}^{a_1\ldots a_t\bar b_1\ldots\bar b_q}\zeta^{\alpha_1}\ldots\zeta^{\alpha_r}\overline{\zeta^{\beta_1}}\ldots\overline{\zeta^{\beta_s}},
$$
where $\zeta=(\zeta^1,\ldots,\zeta^n)$. Then, the reader can check that the tensor is conformally flat if and only if all the polynomials $T^{a_1\ldots a_t\bar b_1\ldots\bar b_q}(\zeta,\bar \zeta)$ are divisible by the Hermitian form $g(\zeta,\bar\zeta):=g_{\alpha\bar\beta}\zeta^\alpha\overline{\zeta^\beta}$. Since
$\nabla g_{\a\bar\b}=0$ (see the second equation of
\eqref{eq-consymmetry}), it is clear that covariant derivatives of a
conformally flat tensor is again conformally flat.

We shall now restrict our attention to the case where the target
manifold $\hat M$ is the standard hyperquadric $Q_{l}^{N}$ in
$\bC\bP^{N+1}$, as defined by \eqref{hyperquadric}. (Thus, in what
follows the CR dimension of $\hat M=Q^N_l$ is $N$.) The crucial
property of the quadric that we shall use is that its Chern-Moser
pseudoconformal curvature tensor $\hat S_{A\bar B C\bar D}$ vanishes
identically. We shall need the following lemma. The corresponding
result in the strictly pseudoconvex case is proved, but not
explicitly stated in \cite{EHZ1}. Although the proof in the general
case is identical to that of the strictly pseudoconvex case, we give
it here for the convenience of reader.

\begin{Lem}\Label{l:codazzi}
  Let $M\subset \bC^{n+1}$ be a smooth
Levi-nondegenerate hypersurface of signature $l\leq n/2$, $f\colon
M\to Q_l^{N}\subset\bC^{N+1}$ a smooth CR mapping that is CR
transversal to $Q^{N}_l$ along $M$, and $\omega_\a{}^a{}_\b$ its
second fundamental form. Then, the covariant derivative tensor
$\omega_\a{}^a{}_\b{}_{;\bar\gamma}$ is conformally flat.
\end{Lem}

\begin{proof} We shall work locally near a point $p\in M$ and use the setup introduced in Section \ref{s:prelim}. Let
\begin{equation}\Label{eq-CRpar}
(\omega,\omega^\a,\omega^{\bar\a},\phi,\phi_{\b}{}^{\a},\phi^{\a},
\phi^{\bar\a},\psi),\quad (\hat \omega,\hat {\omega}^A,\hat
\omega^{\bar A},\hat \phi,\hat \phi_{B}{}^{A},\hat \phi^{A}, \hat
\phi^{\bar A},\hat \psi)
\end{equation}
be the Chern--Moser pseudoconformal connections on the coframe
bundles $Y\to M$ and $\hat Y\to Q^{N}_l$, respectively, pulled back
to $M$ and $Q^{N}_l$ by (the completion of) our admissible coframes
$(\theta,\theta^\a,\theta^{\bar \a})$ and
$(\theta,\theta^A,\theta^{\bar A})$ (see \cite{EHZ1}, Section 3).
The latter connection is then pulled back to $M$ by the embedding
$f$. The $1$-form $\hat \phi_{\a}{}^{a}$ is of the form
\begin{equation}\Label{eq-phialphaa}
\hat\phi_\a{}^a= \omega_\a{}^a{}_\b\theta^\b+\hat D_\a{}^a\theta ,
\end{equation}
for some coefficients $\hat D_\a{}^a$ (see (3.3), (3.6) of \cite{We78} of Proposition 3.1 in
\cite{EHZ1}). By differentiating \eqref{eq-phialphaa}, using the
structure equation for $\hat \phi_{\a}{}^{a}$ ((3.12) in
\cite{EHZ1}; recall that the pseudoconformal curvature $\hat
S_{A\bar B C\bar D}$ of $Q^{N}_l$ vanishes identically), and
identifying the coefficients of
$\theta^\beta\wedge\theta^{\bar\gamma}$, we obtain
\begin{equation}\Label{eq-codazzi2}
\omega_\a{}^a{}_{\b;\bar\gamma}=i(g_{\a\bar\gamma}\hat D_\b{}^a +
g_{\b\bar\gamma}\hat D_\a{}^a),
\end{equation}
which proves the lemma. Here, to simplify the computation, we  choose an
adapted coframe near $p$, the point under study,  such that $\o_\a{}^\b(p)=\hat\o_a{}^b(p)=0$
(cf.\ e.g.\ Lemma 2.1 in \cite{Le88}). We will do the same in the following lemma, too.
\end{proof}

We shall also need the following result that describes how covariant
derivatives commute. A similar result (with a slightly stronger
conclusion) can be found in \cite{EHZ1} (Lemma 7.4). The proof given
there uses a result that does not immediately apply to our current
situation. We give therefore a (more or less) self-contained proof
here.

\begin{Lem}\Label{lem-covdercom}
 Let $M$, $f$, and
$\omega_\a{}^a{}_{\b}$ be as in Lemma $\ref{l:codazzi}$. Then, for
any $s\geq 2$, we have a relation
\begin{equation}\Label{eq-covdercom}
\omega_{\gamma_1}{}^a{}_{\gamma_2;\gamma_3\ldots\gamma_s\a\bar\b}\,
-\,
\omega_{\gamma_1}{}^a{}_{\gamma_2;\gamma_3\ldots\gamma_s\bar\b\a} =
C^a{}_{\gamma_1\ldots\gamma_s\a\bar\b}{}^{\mu_1\ldots\mu_s}{}_b \,
\omega_{\mu_1}{}^b{}_{\mu_2;\mu_3\ldots\mu_s} +
\,T_{\gamma_1\ldots\gamma_s\a\bar\b}{}^a,
\end{equation}
where the tensor
$C^a{}_{\gamma_1\ldots\gamma_s\a\bar\b}{}^{\mu_1\ldots\mu_s}{}_b$
depends only on $(\theta,\theta^\a)$  and the second fundamental
form $\omega_\a{}^a{}_\b$, and
$T_{\gamma_1\ldots\gamma_s\a\bar\b}{}^a$ is conformally flat.
\end{Lem}

\begin{proof} We shall use the pseudoconformal connections in
\eqref{eq-CRpar}, as in the proof of Lemma \ref{l:codazzi} above. By
observing that the left hand side of the identity
\eqref{eq-covdercom} is a tensor, it is enough to show, for each
fixed point $p\in M$, the identity at $p$ with respect to any
particular choice of adapted coframe $(\theta,\theta^A)$ near $p$.
By making a suitable unitary change of coframe $\theta^\a\to
u_\b{}^\a \theta^\b$ and $\theta^a\to u_b{}^a \theta^b$ (in the
tangential and normal directions respectively), we may choose an
adapted coframe near $p$ such that $\o_\a{}^\b(p)=\hat\o_a{}^b(p)=0$
(cf.\ e.g.\ Lemma 2.1 in \cite{Le88}). By using \eqref{eq-cov1} and
\eqref{eq-phialphaa}, we conclude that, relative to this coframe,
the left hand side of \eqref{eq-covdercom} evaluated at $p$ is equal
to, modulo a conformally flat tensor, the coefficient in front of
$\theta^\a\wedge\theta^{\bar\b}$ in the expression
\begin{equation}\Label{eq-exp}
\sum_{j=1}^s
\omega_{\gamma_1}{}^a{}_{\gamma_2;\gamma_3\ldots\gamma_{j-1}\,\mu\,\gamma_{j+1}\gamma_s}\,
d\o_{\gamma_j}{}^\mu -
\omega_{\gamma_1}{}^b{}_{\gamma_2;\gamma_3\ldots\gamma_s}\,
d\hat\phi_b{}^a.
\end{equation}
The first term (i.e.\ the sum) in \eqref{eq-exp} is clearly of the
form on the right hand side of \eqref{eq-covdercom}.  Indeed, the coefficients $d\omega_{\gamma_j}{}^{\mu}$ corresponding to the $C^a{}_{\gamma_1\ldots\gamma_s\a\bar\b}{}^{\mu_1\ldots\mu_s}{}_b$ on the right in \eqref{eq-covdercom} only depend on the coframe $(\theta,\theta^\alpha)$ (and not even on the second fundamental form). It is not clear  that the corresponding coefficients $d\hat\phi_b{}^a$ in the second term of \eqref{eq-exp} depend only on the coframe and the second fundamental form.
To show that it does, we compute $d\hat\phi_b{}^a$ using the structure
equation (3.12) in \cite{EHZ1}, the vanishing of $\theta^a$ on $M$,
and the vanishing of $\hat\phi_\b{}^\a$ and $\hat\phi_b{}^a$ at $p$
modulo $\theta$ to obtain:
\begin{equation}\Label{eq-dphi}
d\hat\phi_b{}^a = \hat\phi_b{}^\mu \wedge \hat\phi_\mu{}^a  -
i\delta_b{}^a \hat\phi_\mu \wedge \theta^\mu  \mod \theta.
\end{equation}
Making use of the fact that $\hat\phi_b{}^\mu=-\overline{\hat\phi_\mu^b}\ \hbox{mod}(\theta)$, we see that
the first term on the right hand side of \eqref{eq-dphi} contributes
the term
$$
g^{\mu\bar\kappa}g_{b\bar
c}\omega_\mu{}^a{}_\gamma\omega_{\bar\kappa}{}^{\bar c}{}_{\bar
\nu},
$$
to the coefficient in front of $\theta^\a\wedge\theta^{\bar\b}$  in \eqref{eq-exp}. We observe that these only depend on the coframe and the second fundamental form.  For the second term on the right
in \eqref{eq-dphi}, we recall from \cite{EHZ1} (see equations (6.1)
and (6.8)) that, pulled back to $M$,
\begin{equation}
\hat \phi^\a=\phi^\a +C_\mu{}^\a\theta^\mu +F^\a\theta
\end{equation}
for some coefficients $C_\mu{}^\a$ and $F^\a$, where
\begin{equation}\Label{eq-Cab}
C_{\a\bar\b} =\frac{i
\omega_\mu{}^a{}_\a\,\omega^\mu{}_a{}_{\bar\b}}{n+2}
-\frac{ig_{\a\bar\b} \omega_\mu{}^a{}_\nu\,\omega^\mu{}_a{}^{\nu}}
{2(n+1)(n+2)}.
\end{equation}
In \eqref{eq-Cab}, we have used the vanishing of the curvature $\hat
S_{A\bar B\nu\bar\mu}$ of the target quadric. We observe that the
coefficients in front of $\theta^\a$ and $\theta^{\bar\b}$ in the
pulled back forms $\hat\phi^\gamma$ are uniquely determined by the coframe
$(\theta,\theta^\a)$ and the scalar products
$\omega_\a{}^a{}_\mu\omega_{\bar\b a\bar\nu}$. Hence, the second
term on the right in \eqref{eq-dphi}, substituted in \eqref{eq-exp}, contributes
only terms of the form that appear on the right hand side of
\eqref{eq-covdercom}. This completes the proof of Lemma
\ref{lem-covdercom}.
\end{proof}

The final ingredient we shall need for the proof of our main result
is the Gauss equation for the second fundamental form of the
embedding. For our purposes, we shall only need the following form
of it. A more general and precise version is stated and proved in
\cite{EHZ1} (Theorem 2.3; the lemma below corresponds to equation
(7.17) in \cite{EHZ1}). The proof in the Levi-nondegenerate case is
identical to that of the strictly pseudoconvex case in \cite{EHZ1},
and is therefore not repeated here.

\begin{Lem}\Label{lem-Gauss}
Let $M$, $f$, and
$\omega_\a{}^a{}_{\b}$ be as in Lemma $\ref{l:codazzi}$. Then,
\begin{equation}\Label{full-gauss}
0= S_{\a\bar \b\mu\bar\nu} + g_{a\bar b}\,\omega_\a{}^a{}_\mu\,
\omega_{\bar\b}{}^{\bar b}{}_{\bar\nu} + T_{\a\bar\b\mu\bar\nu},
\end{equation}
where $S_{\a\bar \b\mu\bar\nu}$ is the Chern-Moser pseudoconformal
curvature of $M$ and $T_{\a\bar\b\mu\bar\nu}$ is a conformally flat
tensor. \end{Lem}

\section{Proof of Theorem $\ref{t:mainhyper}$}

The first step in the proof of Theorem $\ref{t:mainhyper}$ is the
following result concerning the second fundamental form and its
derivatives. The notation is the same as in the previous sections.
(For convenience of notation in the proof, we use $f$ and $\tilde f$
to denote the mappings, rather than $f_0$ and $f$ as in Theorem
$\ref{t:mainhyper}$.) To simplify the notation,  in what follows, we
will use the notation $\omega_\a^a$ , for $a\in \{1,\ldots,N-n\}$,
instead of $\omega_\a^{a+n}$ (and similarly for
$\tilde{\omega}_\alpha^a$.)

\begin{Thm}\Label{t:mainhypertech}
Let $M\subset \bC^{n+1}$ be a smooth
Levi-nondegenerate hypersurface of signature $l\leq n/2$ and $p\in
M$. Let $f\colon M\to Q^{N}_l$ and $\tilde f\colon M\to Q^{\tilde
N}_l$ be smooth CR mappings that are CR transversal to $Q^{N}_l$ at $f(p)$ and
$Q^{\tilde N}_l$ at $\tilde f(p)$, respectively. Suppose that $N-n<l$ and
$\tilde N\geq N$. Fix an admissible coframe $(\theta,\theta^\a)$ on
$M$ and choose corresponding coframes (as given by Proposition
    $\ref{thm-admada}$) $(\theta,\theta^A)_{A=1,\ldots,N}$ and $(\tilde \theta,\tilde
\theta^A)_{A=1,\ldots,\tilde N}$ on $Q^{N}_l$ and $Q^{\tilde N}_l$
adapted to $f(M)$ and $\tilde f(M)$, respectively. Denote by
$(\o_\a{}^a{}_\b)_{a=1,\ldots,N-n}$ and $(\tilde
\o_\a{}^a{}_\b)_{a=1,\ldots,\tilde N-n}$ the second fundamental
forms of $f$ and $\tilde f$, respectively, relative to these
coframes. Let $k\geq 2$ be an integer and assume that the spaces
$E_j(q)$  and $\tilde E_j(q)$, for $2\leq j\leq k$, are of constant
dimension for $q$ near $p$. Then, possibly after a unitary change of
$(\tilde\theta^a)$ near $p$,  the following holds for $2\le j\le k$:
\begin{equation}\Label{cov-eq}
\left\{
\begin{aligned}
& \tilde{\o}_{\gamma_1}{}^a{}_{\gamma_2;\gamma_3,\ldots,\gamma_j} =
\o_{\gamma_1}{}^a{}_{\gamma_2;\gamma_3,\ldots,\gamma_j},\quad a=1,\ldots N-n,\\
& \tilde{\o}_{\gamma_1}{}^i{}_{\gamma_2;\gamma_3,\ldots,\gamma_j}
=0,\qquad i=N-n+1,\ldots, \tilde N-n,
\end{aligned}
\right.
\end{equation}
\end{Thm}

\begin{Rem}\Label{opendense}
 {\rm If $f$ and $\tilde f$ in Theorem \ref{t:mainhypertech} are assumed to be CR transversal to
 $Q^N_l$ and $Q^{\tilde N}_l$ at $f(p)$ and $\tilde f(p)$, respectively, for every $p\in M$, then
 for any $k\geq 2$  the set of points $p\in M$ such that the spaces $E_j(q)$  and $\tilde
E_j(q)$, for $2\leq j\leq k$, are of constant dimension for $q$ near
$p$ is open and dense in $M$.}
\end{Rem}

\begin{proof} Recall the normalization of the Levi forms given by Proposition
$\ref{thm-admada}$. We think of
$(\o_{\gamma_1}{}^a{}_{\gamma_2;\gamma_3,\ldots,\gamma_j})_{a=1\ldots,N-n}$
and
$(\tilde{\o}_{\gamma_1}{}^b{}_{\gamma_2;\gamma_3,\ldots,\gamma_j})_{b=1,\ldots,\tilde
N-n}$ as vectors in $\bC^{N-n}$ and $\bC^{\tilde N-n}$,
respectively. Let $e_j$ denote the dimension of $E_j(q)$, for $q$
near $p$ and $j=2,\ldots, k$. We first make an initial unitary
change of the $\theta^a$, $a=1,\ldots, N-n$, near $p$ such that, for
each $j=2,\ldots, k$, we have
\begin{equation}
\o_{\gamma_1}{}^a{}_{\gamma_2;\gamma_3,\ldots,\gamma_j}=0,\quad
a=e_j+1,\ldots, N-n.
\end{equation}
We then embed $\bC^{N-n}$ in $\bC^{\tilde N-n}$ as the subspace
$\{W\in \bC^{\tilde N-n}\colon W_i=0,\ i=N-n+1,\ldots, \tilde
N-n\}$, i.e.\ we extend
$\o_{\gamma_1}{}^a{}_{\gamma_2;\gamma_3,\ldots,\gamma_j}$ to be 0
for $a=N-n+1,\ldots, \tilde N-n$. The proof now consists of showing
that, possibly after a unitary change of the $\tilde\theta^a$, we
have
\begin{equation}
\tilde{\o}_{\gamma_1}{}^a{}_{\gamma_2;\gamma_3,\ldots,\gamma_j} =
\o_{\gamma_1}{}^a{}_{\gamma_2;\gamma_3,\ldots,\gamma_j},\quad
a=1,\ldots \tilde N-n. \end{equation}
If we subtract the Gauss equations for $\o_\a{}^a{}_\b$ given by  \eqref{full-gauss} from the
corresponding one for $\tilde
\o_\a{}^a{}_\b$, we obtain (since the pseudoconformal curvature
$S_{\a\bar \b\mu\bar\nu}$ in both equations is computed using the same coframe
$(\theta,\theta^\a)$)
\begin{equation}\Label{sub-gauss2}
-\sum_{a=1}^{N-n}\omega_\a{}^a{}_\mu\, \omega_{\bar\b}{}^{\bar
a}{}_{\bar\nu}+ \sum_{b=1}^{\tilde N-n}\tilde\omega_\a{}^b{}_\mu\,
\tilde\omega_{\bar\b}{}^{\bar b}{}_{\bar\nu}=
T'_{\a\bar\b\mu\bar\nu},
\end{equation}
where $T'_{\a\bar\b\mu\bar\nu}$ is a conformally flat tensor. For
brevity, we will write this simply as
\begin{equation}\Label{sub-gauss2half}
-\sum_{a=1}^{N-n}\omega_\a{}^a{}_\mu\, \omega_{\bar\b}{}^{\bar
a}{}_{\bar\nu}+ \sum_{b=1}^{\tilde N-n}\tilde\omega_\a{}^b{}_\mu\,
\tilde\omega_{\bar\b}{}^{\bar b}{}_{\bar\nu}= 0\ \mod \text{{\rm
CFT}}.
\end{equation}
Let $\zeta:=(\zeta^1,\ldots,\zeta^n)$, multiply \eqref{sub-gauss2}
by $\zeta^\a\overline{\zeta^\b}\zeta^\mu\overline{\zeta^\nu}$ and
sum. Since the right hand side of  \eqref{sub-gauss2} is conformally
flat, we obtain (see the beginning of Section \ref{s:SFF})
\begin{equation}\Label{sub-gauss3}
-\sum_{a=1}^{N-n}|\omega^a(\zeta)|^2+ \sum_{b=1}^{\tilde
N-n}|\tilde\omega^b(\zeta)|^2=
A(\zeta,\bar\zeta)\bigg(-\sum_{i=1}^{l}|\zeta^i|^2+
\sum_{j=l+1}^{n}|\zeta^j|^2\bigg),
\end{equation}
where $\omega^a(\zeta)=\omega_\a{}^a{}_\b\zeta^\a\zeta^b$, $\tilde
\omega^b(\zeta)=\tilde\omega_\a{}^b{}_\b\zeta^\a\zeta^b$, and
$A(\zeta,\bar\zeta)$ is a polynomial in $\zeta$ and $\bar\zeta$. Recall that $N-n<l$. By
Lemma \ref{l:main}, we conclude that $A\equiv 0$ and, hence,
\begin{equation}
\sum_{a=1}^{N-n}|\omega^a(\zeta)|^2=\sum_{b=1}^{\tilde
N-n}|\tilde\omega^b(\zeta)|^2,
\end{equation}
or equivalently, since $\omega_\a{}^a{}_\mu=0$ for $a=N-n+1,\ldots, \tilde N-n$,
\begin{equation}\Label{e:sffeq}
\sum_{a=1}^{\tilde N-n}\omega_\a{}^a{}_\mu\, \omega_{\bar\b}{}^{\bar
a}{}_{\bar\nu}= \sum_{b=1}^{\tilde N-n}\tilde\omega_\a{}^b{}_\mu\,
\tilde\omega_{\bar\b}{}^{\bar b}{}_{\bar\nu},
\end{equation}
i.e.\ the collection of vectors $(\omega_\a{}^a{}_\b)_{a=1,\ldots, \tilde N-n}$ and
$(\tilde \omega_\a{}^a{}_\b)_{a=1,\ldots, \tilde N-n}$ have the same scalar products with respect to
the standard scalar product in $\bC^{\tilde N-n}$. Hence, after a
unitary change of $\tilde \theta^a$ (smooth by the constant
dimension assumption on $E_2(q)$), we may assume that
\begin{equation}\Label{e:eq1}
\omega_\a{}^a{}_\b=\tilde\omega_\a{}^a{}_\b
\end{equation}
near $p$.

Next, we take a covariant derivative in the direction
$\theta^{\gamma_1}$ in the Gauss equations for $\omega_\a{}^a{}_\b$ and $\tilde\omega_\a{}^a{}_\b$
respectively, and then subtract the two resulting equations. Since
the covariant derivative of a conformally flat tensor stays
conformally flat and the covariant derivative of the curvature
tensor $S_{\a\bar\b\mu\bar\nu:\gamma}$ is the same in both
equations, we obtain
\begin{equation}\Label{e:sffeq1}
-\sum_{a=1}^{N-n}(\omega_\a{}^a{}_{\mu;\gamma_1}\,
\omega_{\bar\b}{}^{\bar a}{}_{\bar\nu}+\omega_\a{}^a{}_{\mu}\,
\omega_{\bar\b}{}^{\bar a}{}_{\bar\nu;\gamma_1})+ \sum_{b=1}^{\tilde
N-n}(\tilde\omega_\a{}^b{}_{\mu;\gamma_1}\,
\tilde\omega_{\bar\b}{}^{\bar
b}{}_{\bar\nu}+\tilde\omega_\a{}^b{}_\mu\,
\tilde\omega_{\bar\b}{}^{\bar b}{}_{\bar\nu;\gamma_1})=0\quad \mod
\CFT.
\end{equation}
By Lemma \ref{l:codazzi}, the covariant derivatives
$\omega_{\bar\b}{}^{\bar b}{}_{\bar\nu;\gamma_1}$ and
$\tilde\omega_{\bar\b}{}^{\bar b}{}_{\bar\nu;\gamma_1}$ are
conformally flat (since  $\omega_{\bar\b}{}^{\bar
a}{}_{\bar\nu;\gamma_1}=\overline{\omega_{\b}{}^{a}{}_{\nu;\bar\gamma_1}}$).
Hence, by using \eqref{e:eq1}, we obtain
\begin{equation}\Label{e:sffeq2}
\sum_{a=1}^{N-n}(\omega_\a{}^a{}_{\mu;\gamma_1}-\tilde\omega_\a{}^b{}_{\mu;\gamma_1})\,
\omega_{\bar\b}{}^{\bar a}{}_{\bar\nu}=0\quad \mod \CFT.
\end{equation}
Since $N-n<l\leq n/2$, we conclude, by using Lemma \ref{l:main0}
 in the same way we used Lemma \ref{l:main}
above, that in fact
\begin{equation}\Label{e:sffeq21}
\sum_{a=1}^{N-n}(\omega_\a{}^a{}_{\mu;\gamma_1}-\tilde\omega_\a{}^a{}_{\mu;\gamma_1})\,
\omega_{\bar\b}{}^{\bar a}{}_{\bar\nu}=0,
\end{equation}
which in turn implies
\begin{equation}\Label{e:eq2}
\tilde\omega_\a{}^a{}_{\mu;\gamma_1}=\omega_\a{}^a{}_{\mu;\gamma_1},\quad
a=1,\ldots, e_2.
\end{equation}
We now take two covariant derivatives in the directions
$\theta^{\gamma_1}$ and $\theta^{\bar\gamma_1}$ in the two Gauss
equations and subtract the resulting equations. By again using the
facts that covariant derivatives of the form
$\omega_{a}{}^{a}{}_{\b;\bar \gamma}$ are conformally flat and
covariant derivatives of conformally flat tensors are conformally
flat, we obtain
\begin{multline}\Label{e:sffeq3}
-\sum_{a=1}^{N-n}(\omega_\a{}^a{}_{\mu;\gamma_1\bar\gamma_1}\,
\omega_{\bar\b}{}^{\bar
a}{}_{\bar\nu}+\omega_\a{}^a{}_{\mu;\gamma_1}\,
\omega_{\bar\b}{}^{\bar a}{}_{\bar\nu;\bar\gamma_1})+\\
\sum_{b=1}^{\tilde
N-n}(\tilde\omega_\a{}^b{}_{\mu;\gamma_1\bar\gamma_1}\,
\tilde\omega_{\bar\b}{}^{\bar
b}{}_{\bar\nu}+\tilde\omega_\a{}^b{}_{\mu;\gamma_1}\,
\tilde\omega_{\bar\b}{}^{\bar b}{}_{\bar\nu;\bar\gamma_1})=0\quad
\mod \CFT.
\end{multline}
By Lemma \ref{lem-covdercom}, we have
\begin{equation}
\omega_\a{}^a{}_{\mu;\gamma_1\bar\gamma_1}=
\omega_\a{}^a{}_{\mu;\bar\gamma_1\gamma_1}+C^a{}_{\a\b\gamma_1\bar\gamma_1}{}^{\mu\nu}{}_b
\omega_\mu{}^b{}_{\nu}\ \mod\CFT,
\end{equation}
where the $C^a{}_{\a\b\gamma_1\bar\gamma_1}{}^{\mu\nu}{}_b$ only depend
on the coframe $(\theta,\theta^\a)$ and the second fundamental form
$\omega_\a{}^a{}_\b$. Since
$\omega_\a{}^a{}_{\mu;\bar\gamma_1\gamma_1}$ is conformally flat, we
conclude that
\begin{equation}\Label{e:sub1}
\omega_\a{}^a{}_{\mu;\gamma_1\bar\gamma_1}=C^a{}_{\a\b\gamma_1\bar\gamma_1}{}^{\mu\nu}{}_b
\omega_\mu{}^b{}_{\nu}\ \mod\CFT.
\end{equation}
The same argument applied to $\tilde
\omega_\a{}^a{}_{\mu;\gamma_1\bar\gamma_1}$, using the equality
\eqref{e:eq1}, shows that
\begin{equation}\Label{e:sub2}
\tilde \omega_\a{}^a{}_{\mu;\gamma_1\bar\gamma_1}=C^a{}_{\a\b\gamma_1\bar\gamma_1}{}^{\mu\nu}{}_b
\omega_\mu{}^b{}_{\nu}\ \mod\CFT
\end{equation}
with the same $C^a{}_{\a\b\gamma_1\bar\gamma_1}{}^{\mu\nu}{}_b$.
Substituting these identities back in \eqref{e:sffeq3},  we obtain
\begin{equation}\Label{e:sffeq4}
-\sum_{a=1}^{N-n}\omega_\a{}^a{}_{\mu;\gamma_1}\,
\omega_{\bar\b}{}^{\bar a}{}_{\bar\nu;\bar\gamma_1}+
\sum_{b=1}^{\tilde N-n}\tilde\omega_\a{}^b{}_{\mu;\gamma_1}\,
\tilde\omega_{\bar\b}{}^{\bar b}{}_{\bar\nu;\bar\gamma_1}=0\quad
\mod \CFT.
\end{equation}
By using Lemma \ref{l:main} as above, we find that in fact
\begin{equation}\Label{e:sffeq5}
-\sum_{a=1}^{N-n}\omega_\a{}^a{}_{\mu;\gamma_1}\,
\omega_{\bar\b}{}^{\bar a}{}_{\bar\nu;\bar\gamma_1}+
\sum_{b=1}^{\tilde N-n}\tilde\omega_\a{}^b{}_{\mu;\gamma_1}\,
\tilde\omega_{\bar\b}{}^{\bar b}{}_{\bar\nu;\bar\gamma_1}=0.
\end{equation}
Since we already have \eqref{e:eq2}, we conclude that there is a
unitary change of the remaining $\tilde
\theta^{e_2+1},\ldots,\tilde\theta^{\tilde N-n}$ such that
\begin{equation}\Label{e:eq3}
\tilde\omega_\a{}^a{}_{\mu;\gamma_1}=\omega_\a{}^a{}_{\mu;\gamma_1}.
\end{equation}
We notice that such a unitary change of the coframes does not affect \eqref{e:eq1}.

We now complete the proof of Theorem \ref{t:mainhypertech} by
induction, using the ideas above. We assume that
\begin{equation}\Label{e:induct}
\tilde{\o}_{\alpha}{}^a{}_{\beta;\gamma_1,\ldots,\gamma_j} =
\o_{\alpha}{}^a{}_{\beta;\gamma_1,\ldots,\gamma_j},\quad a=1,\ldots
\tilde N-n, \end{equation} holds for all $0\leq j\leq k$ with $k\geq
2$. We wish to prove that \eqref{e:induct} holds for all $0\leq
j\leq k+1$, after possibly another unitary change of the $\tilde
\theta^a$. We apply repeatedly covariant derivatives in the
directions $\theta^{\gamma_1},\ldots,\theta^{\gamma_{k+1}}$ to the
Gauss equations for $\o_{\alpha}{}^a{}_{\beta}$ and $\tilde
\o_{\alpha}{}^a{}_{\beta}$. We obtain, using as above the fact that
$\o_{\alpha}{}^a{}_{\beta;\bar \gamma}$ is conformally flat,
\begin{equation}\Label{dergauss1}
-S_{\a\bar \b\mu\bar\nu;\gamma_1,\ldots,\gamma_{k+1}} =
\sum_{a=1}^{N-n}\omega_\a{}^a{}_{\mu;\gamma_1,\ldots,\gamma_{k+1}}\,
\omega_{\bar\b}{}^{\bar a}{}_{\bar\nu} \ \mod\CFT,
\end{equation}
and
\begin{equation}\Label{dergauss2}
-S_{\a\bar \b\mu\bar\nu;\gamma_1,\ldots,\gamma_{k+1}}
=\sum_{a=1}^{\tilde N-n}\tilde
\omega_\a{}^a{}_{\mu;\gamma_1,\ldots,\gamma_{k+1}}\, \tilde
\omega_{\bar\b}{}^{\bar a}{}_{\bar\nu} \ \mod\CFT,
\end{equation}
Subtracting these two equations, using the fact that
$\o_{\alpha}{}^a{}_{\beta}=\tilde \o_{\alpha}{}^a{}_{\beta}$ and
Lemma \ref{l:main0} as above, we conclude that
\begin{equation}\Label{e:eq2-2}
\tilde\omega_\a{}^a{}_{\mu;\gamma_1\ldots\gamma_{k+1}}=\omega_\a{}^a{}_{\mu;\gamma_1
\ldots\gamma_{k+1}},\quad
a=1,\ldots, e_2.
\end{equation}
We now differentiate the two equations \eqref{dergauss1} and
\eqref{dergauss2} in the direction $\theta^{\bar \lambda_1}$. We
obtain
\begin{multline}
\Label{dergauss1-2}
 -S_{\a\bar
\b\mu\bar\nu;\gamma_1,\ldots,\gamma_{k+1}\bar\lambda_1}=\\
\sum_{a=1}^{N-n}\omega_\a{}^a{}_{\mu;\gamma_1,\ldots,\gamma_{k+1}\bar\lambda_1}\,
\omega_{\bar\b}{}^{\bar a}{}_{\bar\nu}+
\sum_{a=1}^{N-n}\omega_\a{}^a{}_{\mu;\gamma_1,\ldots,\gamma_{k+1}}\,
\omega_{\bar\b}{}^{\bar a}{}_{\bar\nu;\bar\lambda_1} \ \mod\CFT,
\end{multline}
and
\begin{multline}
\Label{dergauss2-2}
-S_{\a\bar
\b\mu\bar\nu;\gamma_1,\ldots,\gamma_{k+1}\bar\lambda_1}=\\
\sum_{a=1}^{\tilde N-n}\tilde
\omega_\a{}^a{}_{\mu;\gamma_1,\ldots,\gamma_{k+1}\bar\lambda_1}\,
\tilde \omega_{\bar\b}{}^{\bar a}{}_{\bar\nu}+ \sum_{a=1}^{\tilde
N-n}\tilde \omega_\a{}^a{}_{\mu;\gamma_1,\ldots,\gamma_{k+1}}\,
\tilde \omega_{\bar\b}{}^{\bar a}{}_{\bar\nu;\bar\lambda_1} \
\mod\CFT,
\end{multline}
We now use Lemma \ref{lem-covdercom} repeatedly to commute the
covariant derivative
$\omega_\a{}^a{}_{\mu;\gamma_1,\ldots,\gamma_{k+1}\bar\lambda_1}$ in
\eqref {dergauss1-2} to the conformally flat derivative
$\omega_\a{}^a{}_{\mu;\bar\lambda_1\gamma_1,\ldots,\gamma_{k+1}}$.
In doing so, we produce, according to Lemma \ref{lem-covdercom}, new
conformally flat terms as well as covariant derivatives of the form
\begin{equation}\Label{newterm1}
(C^a{}_{\a\mu\gamma_1\ldots\gamma_j\bar\lambda_1}{}^{\mu_1\ldots\mu_{j+2}}{}_b
\,
\omega_{\mu_1}{}^b{}_{\mu_2;\mu_3\ldots\mu_{j+2}})_{;\gamma_{j+2}\ldots\gamma_{k+1}}
\end{equation}
with $0\leq j\leq k-1$ and
\begin{equation}\Label{newterm2}
C^a{}_{\a\mu\gamma_1\ldots\gamma_{k}\bar\lambda_1}{}^{\mu_1\ldots\mu_{k+2}}{}_b
\, \omega_{\mu_1}{}^b{}_{\mu_2;\mu_3\ldots\mu_{k+2}}.
\end{equation}
We note that, since
$C^a{}_{\a\mu\gamma_1\ldots\gamma_j\bar\lambda_1}{}^{\mu_1\ldots\mu_{j+2}}{}_b$
only depends on the second fundamental form, all terms of the form
\eqref{newterm1} and \eqref{newterm2} depend only on covariant
derivatives $\omega_\a{}^a{}_{\mu;\gamma_1,\ldots,\gamma_{j}}$ up to
order at most $j=k$. If we repeat this procedure with \eqref
{dergauss2-2}, then the new terms that appear are either conformally
flat or, by the induction hypothesis, precisely the same terms (of
the form \eqref{newterm1} and \eqref{newterm1}) that appear in
\eqref{dergauss1-2}. Hence, when we subtract the two equations
\eqref{dergauss1-2} and \eqref{dergauss2-2} we obtain, using again
the fact that $\tilde \omega_{\bar\b}{}^{\bar
a}{}_{\bar\nu;\bar\lambda_1}=\omega_{\bar\b}{}^{\bar
a}{}_{\bar\nu;\bar\lambda_1}$,
\begin{equation}
\sum_{a=1}^{N-n}(\omega_\a{}^a{}_{\mu;\gamma_1,\ldots,\gamma_{k+1}}-\tilde
\omega_\a{}^a{}_{\mu;\gamma_1,\ldots,\gamma_{k+1}})\,
\omega_{\bar\b}{}^{\bar a}{}_{\bar\nu;\bar\lambda_1} =0. \ \mod\CFT.
\end{equation}
Now, by using Lemma \ref{l:main0} as above we conclude that
\begin{equation}\Label{e:eq2-3}
\tilde\omega_\a{}^a{}_{\mu;\gamma_1\ldots\gamma_{k+1}}=\omega_\a{}^a{}_{\mu;\gamma_1
\ldots\gamma_{k+1}},\quad a=1,\ldots, e_3.
\end{equation}
We now apply repeated derivations in the directions $\theta^{\bar
\lambda_2},\ldots,\theta^{\bar\lambda_{k}}$ to the two equations
\eqref{dergauss1-2} and \eqref{dergauss2-2} and repeat the procedure
and arguments above.  The conclusion is that
\begin{equation}\Label{e:eq2-k}
\tilde\omega_\a{}^a{}_{\mu;\gamma_1\ldots\gamma_{k+1}}=\omega_\a{}^a{}_{\mu;\gamma_1
\ldots\gamma_{k+1}},\quad a=1,\ldots, e_k.
\end{equation}
The details of this are left to the reader.

 In the final step, we
apply a derivation in the direction $\theta^{\bar\lambda_{k+1}}$.
After repeating the procedure above and subtracting the resulting
equations we obtain
\begin{equation}
-\sum_{a=1}^{N-n}\omega_\a{}^a{}_{\mu;\gamma_1,\ldots,\gamma_{k+1}}\,
\omega_{\bar\b}{}^{\bar a}{}_{\bar\nu;\bar\lambda_1\ldots
\bar\lambda_{k+1}}+\sum_{a=1}^{N-n}\tilde
\omega_\a{}^a{}_{\mu;\gamma_1,\ldots,\gamma_{k+1}}\, \tilde
\omega_{\bar\b}{}^{\bar a}{}_{\bar\nu;\bar\lambda_1\ldots
\bar\lambda_{k+1}} =0 \ \mod\CFT.
\end{equation}
We apply Lemma \ref{l:main} as above and conclude that in fact
\begin{equation}
\sum_{a=1}^{N-n}\omega_\a{}^a{}_{\mu;\gamma_1,\ldots,\gamma_{k+1}}\,
\omega_{\bar\b}{}^{\bar a}{}_{\bar\nu;\bar\lambda_1\ldots
\bar\lambda_{k+1}}=\sum_{a=1}^{N-n}\tilde
\omega_\a{}^a{}_{\mu;\gamma_1,\ldots,\gamma_{k+1}}\, \tilde
\omega_{\bar\b}{}^{\bar a}{}_{\bar\nu;\bar\lambda_1\ldots
\bar\lambda_{k+1}} =0.
\end{equation}
It follows, by using also \eqref{e:eq2-k}, that there is a unitary
change of the $\tilde \theta^{a}$, with $a=e_k+1,\ldots,\tilde N-n$,
such that
$$
\tilde \omega_\a{}^a{}_{\mu;\gamma_1,\ldots,\gamma_{k+1}}=
\omega_\a{}^a{}_{\mu;\gamma_1,\ldots,\gamma_{k+1}}.
$$
This completes the induction and, thus, the proof of Theorem
\ref{t:mainhypertech}.
\end{proof}

In view of the definition of constant $(k,s)$-degeneracy given at the end of Section \ref{s:prelim} and Remark \ref{opendense}, we obtain
the following  as a corollary of Theorem \ref{t:mainhypertech}:

\begin{Cor}\Label{c:techhyper}
Let $M$, $p$, $f$, $\tilde f$ be as in Theorem
$\ref{t:mainhypertech}$. Then, there is an open neighborhood $U$ of $p$ in $M$ such that for $q$ in an open dense subset of $U$,
the mapping $\tilde f$ is constantly $(k,s)$-degenerate at $q$ for some $k\geq 2$
and some $s$ with $\tilde N-s\leq N$.
\end{Cor}

To prove Theorem \ref{t:mainhyper}, we also need the
following result (Theorem \ref{thm-main2} below). The corresponding result in the strictly
pseudoconvex case was stated and proved in \cite{EHZ1} (Theorem 2.2
in that paper). The proof in the Levi nondegenerate case is
identical, and is therefore not reproduced here. We embed $\bC^{N+1}$ into $\bC\bP^{N+1}$ in the standard way, i.e. as the open subset $\{[z_0\colon z_1\colon\ldots\colon z_{N+1}]\colon z_0\neq 0\}$.

\begin{Thm}\Label{thm-main2}
 Let $M\subset \bC^{n+1}$ be a real-analytic
connected Levi-nondegenerate hypersurface of signature $l\leq n/2$
and $f\colon M\to Q^{N}_l\cap\bC^{N+1}\subset\bC^{N+1}$ a CR mapping that
is CR transversal to $Q^{N}_l\cap\bC^{N+1}$ at $f(p)$ for $p\in M$. Assume that $f$ is constantly $(k,s)$-degenerate near $p$ for some $k$ and $s$. If $N-n-s<n$, then there is an open neighborhood $V$ of $p$ in $M$ such that $f(V)$ is contained in the
intersection of $Q^{N}_l\cap\bC^{N+1}$ with a complex plane $P\subset \bC^{N+1}$
of codimension $s$.
\end{Thm}

We are now in a position to prove Theorem \ref{t:mainhyper}.

\begin{proof}[Proof of Theorem $\ref{t:mainhyper}$] We first observe that it suffices to show that $f=T\circ L\circ f_0$, where $T$ and $L$ are as in the statement of the theorem, in an open neighborhood of any point $p\in M$. Indeed, if $f=T\circ L\circ f_0$ holds on a nonempty open subset of $U$, then it holds on all of $U$ since both sides are holomorphic mappings $U\to \bC\bP^{N+1}$ and $U$ is connected.

Let $\Pi\subset \bC\bP^{N+1}$ be the hyperplane at infinity, i.e.\ given in homogeneous coordinates by $z_0=0$, and observe that $\Pi$ is biholomorphically equivalent to $\bC\bP^N$. We observe that $Q^N_l\cap \Pi$ is a real hypersurface isomorphic to the hyperquadric $Q^{N-1}_{l-1}$ and, hence, has signature $l-1$. Since $f(U)\not\subset Q^N_l$, it follows from Theorem 5.7 in \cite{BERtrans} that $f(U)$ cannot be contained in $\Pi$. For, if it were, then $f(M)$ would be contained in $Q^N_l\cap\Pi\cong Q^{N-1}_{l-1}$ and $f(U)\not\subset Q^N_l\cap\Pi$, contradicting the conclusion of Theorem 5.7 in \cite{BERtrans}. We claim that there is a dense relatively open subset $M_0\subset M$ such that $f(p)\subset Q^N_l\cap \bC^{N+1}$ and $f$ is CR transversal to $Q^N_l$ at $f(p)$ for every $p\in M_0$. Indeed, the existence of $M_0$ follows from the remarks above and Theorem 1.1 in \cite{BERtrans}, since $M'=Q^N_l$ satisfies condition (1.2) of that theorem ($l\leq n/2\leq n-1$ for $n\geq 1$). A similar argument applies to the mapping $f_0$ and after restricting $M_0$ if necessary, we may assume that $f_0(p)\in Q^{N_0}_l\cap \bC^{N_0+1}$,  $f(p)\in Q^{N}_l\cap \bC^{N+1}$ and that both maps are CR transversal to their target manifolds for every $p\in M_0$.

By Corollary
\ref{c:techhyper} (with the roles of $f$ and $\tilde f$ played by $f_0$ and $f$, respectively), we conclude that there is a nonempty open subset of $M_0$ on which $f$ is constantly $(k,s)$-degenerate for some $k$ and $s$ with $N-s<N_0$. Since $N-n-s\leq N_0-n<l\leq n/2<n$, Theorem \ref{thm-main2} implies that there exists a point $p_0\in M_0$ and an open neighborhood $V$ of $p_0$ in $M_0$ such that $f(V)$ is contained
in the intersection of $Q^{N}_l\cap \bC^{N+1}$ with a complex plane $P\subset
\bC^{N+1}$ of codimension $s$. Since  $N-s<N_0$, $P$ is of
dimension $\leq N_0+1$. Without loss of  generality (by enlarging $P$ if necessary), we may assume that the dimension of $P$ equals $N_0+1$.
Since $f$ is CR transversal to
$Q^{N}_l$ at $f(p_0)$, the plane $P$ must be transversal to $Q^{N}_l$ at $f(p_0)$. The
intersection $Q^{N}_l\cap P$ is again a hyperquadric (inside $P$)
and its signature cannot exceed $l$. Since $f\colon V\to Q^{N}_l\cap
P$ is a CR mapping that is CR transversal to $Q^{N}_l\cap P$, we
conclude that $Q^{N}_l\cap P$ is a hyperquadric whose signature cannot be less than $l$, and hence the signature of $Q^{N}_l\cap P$ equals $l$. Let $\tilde P$ be the projective plane in $\bC\bP^{N+1}$ whose restriction to $\bC^{N+1}$ is $P$. Also, let $\tilde P_0$ denote the projective plane of dimension $N_0+1$ given by
$$\tilde P_0:=\{[z_0\colon z_1\colon \ldots\colon z_{N+1}]\in \bC\bP^{N+1}\colon z_{N_0+2}=\ldots =z_{N+1}=0\}.$$
Since both intersections $\tilde P\cap Q^N_l$ and $\tilde P_0\cap Q^N_l$ are nondegenerate quadrics of signature $l$, there exists (by elementary linear algebra) an automorphism  $S\in \Aut(Q^N_l)$ such that $S(\tilde P)=\tilde P_0$.  Hence, the holomorphic mappings $S\circ f$ and $L\circ f_0$, where $L$ is the linear embedding given by \eqref{linear}, both send $V$ (by further shrinking $V$ if necessary) into the nondegenerate quadric of signature $l$ in the $(N_0+1)$-dimensional subspace $ \{z_{N_0+2}=\ldots= z_{N+1}=0\}\subset \bC^{N+1}$, which we may identify with the hyperquadric $Q^{N_0}_l\cap \bC^{N_0+1}$ in $\bC^{N_0+1}$.
Now, since $(N_0-n)+(N_0-n)<2l\leq n$ and $M$ is not locally equivalent to the quadric $Q^n_{n/2}$, Theorem 1.6  in \cite{EHZ2} implies that there is an automorphism $T'\in \Aut(Q^{N_0}_l)$  such that $S\circ f= L\circ T'\circ f_0$. Hence, near $p_0$, we have $f=S^{-1}\circ L\circ T'\circ f_0$. The mapping $S^{-1}\circ L\circ T'$ is a holomorphic embedding $\bC\bP^{N_0+1}\to \bC\bP^{N+1}$ that sends $Q^{N_0}_l$ into $Q^{N}_l$. It follows from the hypotheses that the signature $l$ of the quadric $Q^{N_0}_l$ cannot be $N_0/2$ and, hence, it follows from \cite{BH} that there is an automorphism $T\in\Aut(Q^N_l)$ such that $S^{-1}\circ L\circ T'=T\circ L$. Consequently, the identity $f=T\circ L\circ f_0$ holds in a neighborhood of $p_0$ in $\bC^{n+1}$.  This completes the proof of Theorem \ref{t:mainhyper} in view of the remark at the beginning of the proof.\end{proof}

\begin{Rem} {\rm The proof of Theorem \ref{t:mainhyper} could also
be completed without reference to \cite{EHZ2} by suitably modifying
the proof of Theorem 7.2 in \cite{EHZ1} to the Levi nondegenerate
situation. }
\end{Rem}


\begin{thebibliography}{EHZ04}

\bibitem[A74]{A} Alexander, H.: Holomorphic mappings from the ball and polydisc.
{\em Math. Ann.}  {\bf 209},  (1974), 249--256.

\bibitem[BER06]{BERtrans} Baouendi, M. S., Ebenfelt, P., Rothschild, L. P.:
Transversality of holomorphic mappings between real hypersurfaces in
different dimensions. {\em Comm. Anal. Geom.}, (to appear), {\tt
http://front.math.ucdavis.edu/math.CV/0701432}

\bibitem [BH05]{BH} S. Baouendi; X. Huang: Super-rigidity for holomorphic
mappings between hyperquadrics with positive signature.  {\em J.
Differential Geom.}, {\bf  69},  (2005),  no. 2, 379--398.


\bibitem[CM74]{CM} Chern, S.-S.; Moser, J. K.:
Real hypersurfaces in complex manifolds. {\em Acta Math. } {\bf
133}, 219--271, (1974).

\bibitem[CS83]{CS83}
[CS83] J. Cima and T. J. Suffridge, A reflection principle with
applications to proper holomorphic mappings, {\em Math Ann.} {\bf
265}, 489-500 (1983).

\bibitem[CS90]{CS90} Cima, J. A.; Suffridge, T. J.: Boundary behavior of rational
proper maps. {\em Duke Math. J.} {\bf  60},  (1990),  no. 1,
135--138.

\bibitem  [Da93] {DA} J. D'Angelo, Several Complex Variables and the
Geometry of Real Hypersurfaces, CRC Press, Boca Raton, 1993.



\bibitem[EHZ04]{EHZ1}
Ebenfelt, P.; Huang, X.; Zaitsev, D.: Rigidity of CR-immersions into
spheres. {\em Comm. Anal. Geom.}, {\bf  12} (2004), no. 3, 631--670.

\bibitem[EHZ05]{EHZ2}
Ebenfelt, P.; Huang, X.; Zaitsev, D.: The equivalence problem and
rigidity for hypersurfaces embedded into hyperquadrics. {\em Amer.
J. Math.}, {\bf 127} (2005),  169--191.

\bibitem[Fa82]{Fa82} Faran, J. J.: Maps from the two-ball to the three-ball.
{\em Invent. Math.}  {\bf 68},  (1982), no. 3, 441--475.

\bibitem [Fa86]{Faran86} Faran, J. J.: On the linearity of proper maps
between balls in the lower codimensional case. {\em J. Differential
Geom.} {\bf 24}, 15-17, (1986).

\bibitem [Fo86]{Fo86} F. Forstneric, Proper holomorphic maps from balls, Duke Math. J. 53, 1986, 427-441.




\bibitem [Ha05]{Ha05} H. Hamada, Rational proper holomorphic maps from ${\bf{B}}^n$
into
${\bf{B}}^{2n}$,  {\em Math. Ann.}  {\bf 331} (no.3), 693--711, 2005.



\bibitem[H\"o90]{Ho} H\"ormander, L.: {\em An introduction to complex analysis in several variables.
Third edition.} North-Holland Mathematical Library, 7. North-Holland
Publishing Co., Amsterdam, 1990.



\bibitem [Hu99]{Hu1} Huang, X.:
On a linearity problem for proper holomorphic maps between balls in
complex spaces of different dimensions. {\em J. Differential Geom.}
{\bf 51}, 13--33, (1999).

\bibitem[HJ01]{HJ} Huang, X.; Ji, S.: Mapping $\bold B\sp n$ into $\bold
B\sp {2n-1}$.  {\em Invent. Math.}  {\bf 145}, (2001),  no. 2,
219--250.

\bibitem[HJX06]{HJX} Huang, X.; Ji, S.; Xu, D.: A new gap phenomenon for proper holomorphic mappings from $B\sp n$ into $B\sp N$. {\em Math. Res. Lett.},  {\bf 13 }, (2006), no. 4, 515--529.

\bibitem[La01]{Lamel} Lamel,~B.: A reflection principle for real-analytic submanifolds of complex spaces.
{\em J. Geom. Anal.} {\bf 11}, no. 4, 625--631, (2001).


\bibitem [Le88]{Le88} Lee, J. M.: Pseudo-Einstein structures on CR
  manifolds. {\em Amer. J. Math.} {\bf 110}, 157--178, (1988).

\bibitem[Le56]{Lewy} Lewy, H.: On the local character of the solutions of an atypical linear
differential equation in three variables and a related theorem for
regular functions of two complex variables.  {\em Ann. of Math.}
{\bf 64}, (1956), 514--522.

\bibitem[M67]{Mal} Malgrange, B.: {\em Ideals of differentiable functions.} Tata Institute of
Fundamental Research Studies in Mathematics, No. 3 Tata Institute of
Fundamental Research, Bombay; Oxford University Press, London 1967.


\bibitem[T75]{Ta75} Tanaka,~N.: {\em A differential geometric study on strongly pseudo-convex manifolds.}
Lectures in Mathematics, Department of Mathematics, Kyoto
University, No. 9. Kinokuniya Book-Store Co., Ltd., Tokyo, 1975.


\bibitem[W78]{We78} Webster, S. M.: Pseudohermitian structures on a
  real hypersurface. {\em J. Differential Geom.} {\bf 13}, 25--41, (1978).

\bibitem[W79]{We79} Webster, S. M.: The rigidity of C-R hypersurfaces
in a sphere. {\em Indiana Univ. Math. J.} {\bf 28}, 405--416,
(1979).

\end{thebibliography}
\end{document}